\documentclass[12pt,a4paper]{article}
\usepackage{amsfonts}
\usepackage{amssymb}
\usepackage{mathrsfs}
\usepackage{amsmath}
\usepackage{amsthm}
\usepackage{enumerate}
\usepackage{cite}
\newtheorem{defn}{Definition}[section]
\newtheorem{thm}[defn]{Theorem}
\newtheorem{lem}[defn]{Lemma}
\newtheorem{prop}[defn]{Proposition}
\newtheorem{cor}[defn]{Corollary}
\newtheorem{ex}[defn]{Example}
\newtheorem{re}[defn]{Remark}
\def\K{{\bf K}}

\def\ad{{{\rm ad}}}
\def\Der{{{\rm Der}}}

\def\Id{{{\rm Id}}}
\def\inner{{{\rm inner}}}

\begin{document}
\title{{\bf Representations and module-extensions of hom 3-Lie algebras}}
\author{\normalsize \bf Yan Liu,  Liangyun Chen, Yao Ma}
\date{{{\small{ School of Mathematics and Statistics,
  Northeast Normal University,\\
   Changchun 130024, China
 }}}} \maketitle
\date{}

 {\bf\begin{center}{Abstract}\end{center}}

In this paper, we study the representations and module-extensions of
hom 3-Lie algebras. We show that a linear map between hom 3-Lie
algebras is a morphism if and only if its graph is a hom 3-Lie
subalgebra and show that the derivations of  a hom 3-Lie algebra is
a Lie algebra.
 Derivation extension of hom 3-Lie algebras are also studied as an
application. Moreover, we introduce the definition of
$T_{\theta}$-extensions and $T^{*}_{\theta}$-extensions of hom 3-Lie
sub-algebras in terms of modules,  provide the necessary and
sufficient conditions for $2k$-dimensional metric hom 3-Lie algebra
to be isomorphic to a $T^{*}_{\theta}$-extensions.

\renewcommand{\thefootnote}{\fnsymbol{footnote}}
\footnote[0]{ Corresponding author(L. Chen): chenly640@nenu.edu.cn.}
\footnote[0]{ Supported by  NNSF of China (No. 11171057),  Natural
Science Foundation of  Jilin province (No. 201115006), Scientific
Research Foundation for Returned Scholars
    Ministry of Education of China and the Fundamental Research Funds for the Central Universities. }

\section{Introduction}

3-Lie algebras are special types of n-Lie algebras and have close
relationships with many important fields in mathematics and
mathematical physics[6,7,10]. The structure of 3-Lie algebras is
closely linked to the supersymmetry and gauge symmetry
transformations of the world-volume theory of multiple coincident
M2-branes and is applied to the study of the Bagger-Lambert theory.
Moreover, the $n$-Jacobi identity can be regarded as a generalized
Plucker relation in the physics literature.

A Hom-Lie algebra is a triple $(L,[\cdot,\cdot]_{L},\alpha)$, where
$\alpha$ is a linear self-map, in which the skewsymmetric bracket
satisfies an $\alpha$-twisted variant of the Jacobi identity, called
the Hom-Jacobi identity. When $\alpha$ is the identity map, the
Hom-Jacobi identity reduces to the usual Jacobi identity, and $L$ is
a Lie algebra. The notion of  hom-Lie algebras was introduced by
Hartwig, Larsson and Silvestrov to describe the structures on
certain deformations of the Witt algebra and the Virasoro algebra
[12]. Hom-Lie algebras are also related to deformed vector fields,
the various versions of the Yang-Baxter equations, braid group
representations, and quantum groups [7,18,19].

In this paper we give the definition of hom 3-Lie algebra, and
construct new hom 3-Lie algebras
 from an existing hom 3-Lie algebra $(L, [\cdot, \cdot, \cdot]_L, \alpha)$ by adding $L$-module $A$ to $L$. Specializing to the
 case $A=L^{*}$, the dual space of $L$, we obtain the analogue of $T^{*}$-extension for hom
 3-Lie algebras. We then investigate their relationships to even dimensional metric hom 3-Lie algebras.

The paper is organized as follow. In Section 2 after giving the definition of
  hom 3-Lie algebras,  we show that the direct sum of two  hom 3-Lie algebras
  is still a  hom 3-Lie algebra.  A linear map between  hom 3-Lie algebras  is
  a morphism if and only if its graph is a  hom 3-Lie subalgebra. In section 3
  we study derivations of multiplicative  hom 3-Lie algebras. For any nonnegative
   integer $k$,  we define $\alpha^{k}$-derivations of multiplicative  hom 3-Lie
   algebras. Considering the direct sum of the space of $\alpha$-derivations,
   we prove that it is a Lie algebra (Proposition 3.3). In particular,  any $\alpha$-derivations
   gives rise to a derivation extension of the multiplicative  hom 3-Lie algebras
   $(L, [\cdot, \cdot, \cdot]_L, \alpha)$(Theorem 3.5). In Section 4 we give the
   definition of representations of multiplicative hom 3-Lie algebras. We show that
   one can obtain the semidirect product multiplicative hom 3-Lie algebras
   $(L\oplus V, [\cdot, \cdot, \cdot]_{\rho A}, \alpha+A)$ associated to any
   representation $\rho A$ on $V$ of the multiplicative hom 3-Lie algebras
   $(L, [\cdot, \cdot, \cdot]_L, \alpha)$ (Proposition 4.2). We also describes
    module extensions (associated with a 3-cocycle) of hom 3-Lie algebras.
    In Section 5,  $T_{\theta}^{*}$-extensions of  hom 3-Lie algebras is studied.
\section{Hom 3-Lie algebra}
\begin{defn}
 $(1)$\, A hom 3-Lie algebra is a triple  $(L,[\cdot,\cdot,\cdot]_{L},\alpha)$ consisting of a vector space $L$, a 3-ary skew-symmetric operation
$[\cdot,\cdot,\cdot]_L:{\wedge}^3L\rightarrow L$ and a linear map $\alpha:L\rightarrow L$ satisfying
\begin{eqnarray}[\alpha(x),\alpha(y),[u,v,w]_{L}]_{L}&=&[[x,y,u]_{L},\alpha(v),\alpha(w)]_{L}+[\alpha(u),[x,y,v]_{L},\alpha(w)]_{L}\notag\\
&&+[\alpha(u),\alpha(v),[x,y,w]_{L}]_{L}.
\end{eqnarray}
where $x,y,u,v,w\in L$, the identity (1) is called the hom-Jacobi
identity.

$(2)$\, A hom 3-Lie algebra is called a multiplicative hom 3-lie
algebra if $\alpha$ is an algebraic morphism, i.e. for any $x,
y, z \in L$, we have
 $\alpha([x,y,z]_L)=[\alpha(x),\alpha(y),\alpha(z)]_L.$

 $(3)$\, A hom 3-Lie algebra is called a regular hom 3-Lie algebra if $\alpha$ is an algebra
 automorphism.

$(4)$\,  A sub-vector space $I\subseteq L$ is a hom 3-Lie
sub-algebra of $(L,[\cdot,\cdot,\cdot]_{L},\alpha)$  if
$\alpha(I)\subseteq I$  and $I$ is closed under the bracket
operation  $[\cdot,\cdot,\cdot]_L$, i.e. $[I,I,I]_L\in I$. $I$ is
called a hom 3-Lie algebra ideal of $L$  if $[I,L,L]_L\subseteq I$.
 \end{defn}

  If $\alpha ={\rm Id}$ in Definition 2.1, then  $(L,[\cdot,\cdot,\cdot]_{L},\alpha)$  is a  3-Lie
  algebra [8].

 Consider the  direct sum of two hom 3-Lie algebras, we have
 \begin{prop}\label{proposition2.1}
Given two hom 3-Lie algebras $(L, [\cdot,\cdot,\cdot]_L, \alpha)$ and $(\Gamma, [\cdot,\cdot, \cdot]_\Gamma, \beta)$,  there is a  hom 3-Lie algebra $(L\oplus\Gamma, [\cdot,\cdot,\cdot]_{L\oplus\Gamma}, \alpha+\beta)$,  where the 3-ary skew-symmetric operation $[\cdot, \cdot, \cdot]_{L\oplus\Gamma}:{\wedge}^2(L\oplus\Gamma\rightarrow L\oplus\Gamma)$ is given by
$${[u_1+v_1,u_2+v_2,u_3+v_3]}_{L\oplus\Gamma}={[u_1, u_2,u_3]}_L+{[v_1, v_2,v_3]}_\Gamma,  \forall  u_i\in L,  v_i\in \Gamma ,$$
and the linear map $(\alpha+\beta):L\oplus\Gamma \rightarrow L\oplus\Gamma$ is given by
$$(\alpha+\beta)(u+v)=\alpha(u)+\beta(v), \forall  u\in L,  v\in \Gamma.$$
 \end{prop}
\begin{proof} It is clear by Definition 2.1.\end{proof}
\begin{defn}Let $(L, [\cdot, \cdot, \cdot]_L, \alpha)$ and $(\Gamma, [\cdot, \cdot, \cdot]_\Gamma, \beta)$ be two hom 3-Lie algebras.
 A morphism $\phi:L \rightarrow \Gamma$ is said to be a morphism of hom 3-Lie algebras if
\begin{equation} \phi[u, v, w]_L=[\phi(u), \phi(v),\phi(w)]_\Gamma, \forall  u, v,w\in L, \end{equation}
\begin{equation}\phi\circ\alpha=\alpha\circ\phi. \end{equation}

Denote by $\mathfrak{G}_\phi\in L\oplus\Gamma$ be the graph of a
linear map $\phi:L \rightarrow \Gamma$.
\end{defn}

\begin{prop}\label{proposition2.1}
A linear map $\phi:(L, [\cdot, \cdot, \cdot]_L, \alpha)\rightarrow
(\Gamma, [\cdot, \cdot, \cdot]_\Gamma, \beta)$ is a morphism of hom
3-Lie algebras if and only if the graph $\mathfrak{G}_\phi\subset
L\oplus\Gamma$ is a hom 3-Lie sub-algebras of  $(L\oplus\Gamma,
[\cdot, \cdot, \cdot]_{L\oplus\Gamma}, \alpha+\beta)$.
\end{prop}

\begin{proof}
Let $\phi:(L, [\cdot, \cdot, \cdot]_L, \alpha)\rightarrow(\Gamma,
[\cdot, \cdot, \cdot]_\Gamma, \beta)$ be a morphism of hom 3-Lie
algebras. For any $u, v, w\in L$, we have
$$[u+\phi(u), v+\phi(v), w+\phi(w)]_{L\oplus\Gamma}=[u, v, w]_{L}+[\phi(u), \phi(v),\phi(w)]_{\Gamma}=[u, v,w]_{L}+\phi[u, v,w]_L.$$
Then the graph $\mathfrak{G}_\phi$ is closed under the bracket
operation $[\cdot, \cdot, \cdot]_{L\oplus\Gamma}.$ Furthermore, by
(2), we have
$$(\alpha+\beta)(u+\phi(u))=\alpha(u)+\beta\circ\phi(u)=\alpha(u)+\phi\circ\alpha(u),$$
which implies that $(\alpha+\beta)(\mathfrak{G}_\phi)\subset
\mathfrak{G}_\phi.$ Thus $\mathfrak{G}_\phi$ is a hom 3-Lie algebras
of $$(L\oplus\Gamma, [\cdot, \cdot, \cdot]_{L\oplus\Gamma},
\alpha+\beta).$$

Conversely,  if the graph $\mathfrak{G}_\phi\subset L\oplus\Gamma$
is a hom 3-Lie subalgebras of $(L\oplus\Gamma, [\cdot, \cdot,
\cdot]_{L\oplus\Gamma}, \alpha+\beta)$, then we have
$$[u+\phi(u), v+\phi(v),w+\phi(w)]_{L\oplus\Gamma}=[u, v,w]_{L}+[\phi(u), \phi(v),\phi(w)]_{\Gamma}\in\mathfrak{G}_\phi, $$
which implies that $$[\phi(u), \phi(v),\phi(w)]_{\Gamma}=\phi[u, v,w]_{L}.$$
Furthermore,  $(\alpha+\beta)(\mathfrak{G}_{\phi})\subset \mathfrak{G}_{\phi}$ yields that
$$(\alpha+\beta)(u+\phi(u))=\alpha(u)+\beta\circ\phi(u)\in\mathfrak{G}_\phi, $$
which is equivalent to the condition $\beta\circ\phi(u)=\phi\circ\alpha(u)$,  i.e. $\beta\circ\phi=\phi\circ\alpha$. Therefore,  $\mathfrak{G}_{\phi}$ is a morphism of hom 3-Lie algebras.
\end{proof}

\section{Derivations of hom 3-Lie algebras}
Let $(L, [\cdot, \cdot, \cdot]_L, \alpha)$ be a multiplicative hom 3-Lie algebra. For any nonnegative integer $k$, denote by $\alpha^k$ the $k$-times composition of $\alpha$,  i.e.$$\alpha^k=\alpha\circ\cdots\circ\alpha \quad(k-times).$$
In particular,  $\alpha^0=\Id$ and $\alpha^1=\alpha.$ If $(L, [\cdot, \cdot, \cdot]_L, \alpha)$ is a regular hom 3-Lie algebra, we denote by $\alpha^{-k}$ the $k$-times composition of $\alpha^1, $ the inverse of $\alpha.$
\begin{defn}
For any nonnegative integer $k$, a linear map $D:L \rightarrow L$ is called an $\alpha^k$-derivation of the multiplicative  hom 3-Lie algebra $(L, [\cdot, \cdot, \cdot]_L, \alpha)$, if
\begin{equation} [D, \alpha]=0, \quad i.e.\quad{D}\circ\alpha=\alpha\circ{D}, \end{equation}
and
\begin{eqnarray}D[u, v,w]_L&=&[D(u), \alpha^{k}(v),\alpha^{k}(w)]_L+[\alpha^{k}(u), D(v),\alpha^{k}(w)]_L\notag\\
&&+[\alpha^{k}(u),\alpha^{k}(v), D(w)]_L,  \forall u, v, w\in
L.\end{eqnarray} For a regular hom 3-Lie algebra,
$\alpha^{-k}$-derivations can be defined similarly.

Denote by $\Der_a^{k}(L)$ the set of $\alpha^k$-derivations of the multiplicative  hom 3-Lie algebra  $(L, [\cdot, \cdot, \cdot]_L, \alpha)$. For any $u_{1},u_{2}\in L$ satisfying $\alpha(u_{1})=u_{1},\alpha(u_{2})=u_{2}$,  define $\ad_{k}(u_{1},u_{2}):L \rightarrow L$ by\\
$$\ad_{k}(u_{1},u_{2})(v)=[u_{1},u_{2}, \alpha^{k}(v)]_L,   \forall v\in L.$$
Then $\ad_{k}(u_{1},u_{2})$ is an $\alpha^{k+1}$-derivation, which we call an \textbf{\inner} $\alpha^{k+1}$-derivation. In fact,  we have
$$\ad_{k}(u_{1},u_{2})(\alpha(v))=[u_{1},u_{2}, \alpha^{k+1}(v)]_L=\alpha([u_{1},u_{2}, \alpha^{k}(v)]_L)=\alpha\circ \ad_{k}(u_{1},u_{2})(v), $$
which implies that (4) in Defition 3.1 is satisfied. On the other hand,  we have
\begin{eqnarray*}&&\ad_{k}(u_{1},u_{2})([v_{1},v_{2},v_{3}]_L)\\
&=&[u_{1},u_{2}, \alpha^{k}([v_{1},v_{2},v_{3}]_L)]_L=[\alpha(u_{1}),\alpha(u_{2}), [\alpha^{k}(v_{1}), \alpha^{k}(v_{2},\alpha^{k}(v_{3}]_L]_L\\
&=&[[u_{1},u_{2},\alpha^{k}(v_{1})]_{L},\alpha^{k+1}(v_{2}),\alpha^{k+1}(v_{3})]_L+[\alpha^{k+1}(v_{1}),[u_{1},u_{2},\alpha^{k}(v_{2})]_{L},\alpha^{k+1}(v_{3})]_L\\
&&+[\alpha^{k+1}(v_{1}),\alpha^{k+1}(v_{2}),[u_{1},u_{2},\alpha^{k}(v_{3})]_{L}]_L\\
&=&[\ad_{k}(u_{1},u_{2})(v_{1}),\alpha^{k+1}(v_{2}),\alpha^{k+1}(v_{3})]_L+[\alpha^{k+1}(v_{1}),\ad_{k}(u_{1},u_{2})(v_{2}),\alpha^{k+1}(v_{3})]_L\\
&&+[\alpha^{k+1}(v_{1}),\alpha^{k+1}(v_{2}),\ad_{k}(u_{1},u_{2})(v_{3})]_L.
\end{eqnarray*}
Therefore,  $\ad_{k}(u_{1},u_{2})$ is an  $\alpha^{k+1}$-derivation.
Denote by ${\rm Inn}_{\alpha^{k}}(L)$ the set of inner
$\alpha^{k}$-derivations, i.e.
\begin{equation} {\rm Inn}_{\alpha^{k}}(L)=\{[u_{1},u_{2},\alpha^{k-1}(\cdot)]_L|u_{1},u_{2}
\in L, \alpha(u_{i})=(u_{i}),i=1,2\}.\end{equation} In particular we
use $\ad$ represent $\ad_{0}$.
\end{defn}
For any $D \in \Der_{\alpha^{k}}(L)$ and $D^{'} \in \Der_{\alpha^{s}}(L)$, define their commutator $[D, D^{'}]$ as usual:
\begin{equation}[D, D^{'}]=D \circ D^{'}-D^{'} \circ D.\end{equation}
\begin{lem}For any $D \in \Der_{\alpha^{k}}(L)$ and $D^{'} \in \Der_{\alpha^{s}}(L)$, we have
$$[D, D^{'}]\in \Der_{\alpha^{k+s}}(L).$$
\end{lem}
\begin{proof}
For any $u, v, w \in L,$ we have
\begin{eqnarray*}&&[D, D^{'}]([u, v,w]_L)=D\circ D^{'}([u, v,w]_L)-D^{'}\circ D([u, v,w]_L)\\
&=&D([D^{'}(u), \alpha^{s}(v),\alpha^{s}(w)]_L+[\alpha^{s}(u), D^{'}(v),\alpha^{s}(w)]_{L}+[\alpha^{s}(u),\alpha^{s}(v) ,D^{'}(w)]_{L})\\
&&-D^{'}([D(u), \alpha^{k}(v),\alpha^{k}(w)]_L+[\alpha^{k}(u), D(v),\alpha^{k}(w)]_{L}+[\alpha^{k}(u),\alpha^{k}(v) ,D(w)]_{L})\\
&=&[D\circ D^{'}(u), \alpha^{k+s}(v),\alpha^{k+s}(w)]_L+[\alpha^{k}\circ D^{'}(u), D\circ\alpha^{s}(v),\alpha^{k+s}(w)]_L\\
&&+[\alpha^{k}\circ D^{'}(u),\alpha^{k+s}(v), D\circ\alpha^{s}(w)]_L+[D\circ \alpha^{s}(u), \alpha^{k} \circ D^{'}(v),\alpha^{k+s}(w)]_L\\
&&+[\alpha^{k+s}(u), D \circ D^{'}(v),\alpha^{k+s}(w)]_L+[\alpha^{k+s}(u),\alpha^{k}\circ D^{'}(v), D\circ \alpha^{s}(w)]_L\\
&&+[D\circ \alpha^{s}(u), \alpha^{k+s}(v),\alpha^{k}\circ D^{'}(w)]_L+[\alpha^{k+s}(u),D\circ \alpha^{s}(v),\alpha^{k}\circ D^{'}(w)]_L\\
&&+[ \alpha^{k+s}(u),\alpha^{k+s}(v),D\circ D^{'}(w)]_L-[D^{'}\circ D(u), \alpha^{k+s}(v),\alpha^{k+s}(w)]_{L}\\
&&-[\alpha^{s}\circ D(u), D^{'}\circ\alpha^{k}(v),\alpha^{k+s}(w)]_L-[\alpha^{s}\circ D(u),\alpha^{k+s}(v), D^{'}\circ\alpha^{k}(w)]_L\\
&&-[D^{'}\circ \alpha^{k}(u), \alpha^{s}\circ D(v),\alpha^{k+s}(w)]_L-[\alpha^{k+s}(u), D^{'} \circ D(v),\alpha^{k+s}(w)]_L\\
&&-[\alpha^{k+s}(u),\alpha^{s}\circ D(v), D^{'}\circ \alpha^{k}(w)]_L-[D^{'}\circ \alpha^{k}(u), \alpha^{k+s}(v),\alpha^{s}\circ D(w)]_L\\
&&-[ \alpha^{k+s}(u),D^{'}\circ \alpha^{k}(v),\alpha^{s}\circ D(w)]_L-[ \alpha^{k+s}(u),\alpha^{k+s}(v),D^{'}\circ D(w)]_L.
\end{eqnarray*}
Since D and $D^{'}$ satisfy
$$D\circ \alpha=\alpha \circ D, D^{'} \circ \alpha=\alpha \circ D^{'}, $$
we have
$$D\circ \alpha^{s}=\alpha^{s} \circ D, D^{'} \circ \alpha^{k}=\alpha^{k} \circ D^{'}. $$
Then
\begin{eqnarray*}
&&[D, D^{'}]([u, v,w]_L)\\
&=&[D\circ D^{'}(u)-D^{'}\circ D(u), \alpha^{k+s}(v),\alpha^{k+s}(w)]_L\\
&&+[\alpha^{k+s}(u),D\circ D^{'}(v)-D^{'}\circ D(v), \alpha^{k+s}(w)]_L\\
&&+[\alpha^{k+s}(u),\alpha^{k+s}(v),D\circ D^{'}(w)-D^{'}\circ D(w)]_L\\
&=&[[D, D^{'}](u), \alpha^{k+s}(v),\alpha^{k+s}(w)]_L+[\alpha^{k+s}(u),[D, D^{'}](v), \alpha^{k+s}(w)]_L\\
&&+[\alpha^{k+s}(u),\alpha^{k+s}(v),[D, D^{'}](w)]_L.
\end{eqnarray*}
Furthermore, it is straightforward to see that
\begin{eqnarray*}
[D, D^{'}]\circ \alpha&=&D\circ D^{'}\circ\alpha-D^{'}\circ D\circ\alpha\\
&=&\alpha\circ D\circ D^{'}-\alpha\circ D^{'}\circ D\\
&=&\alpha \circ [D, D^{'}],
\end{eqnarray*}
which yields that $[D, D^{'}]\in \Der_{\alpha^{k+s}}(L)$.
\end{proof}
Denote by
 \begin{eqnarray}
 \Der(L)=\oplus_{k\geq 0} \Der_{\alpha^{k}}(L).
 \end{eqnarray}

By Lemma 3.2, we have
\begin{prop}With the above notations,  \Der(L) is a Lie algebra,  in which the Lie bracket is given by (7).\end{prop}
\begin{re}Similarly,  we can obtain a Lie algebra  $\oplus _{k}\Der_{\alpha^{k}}(L)$,  where $k$ is any integer, $L$ is a regular hom 3-Lie algebra.
\end{re}

At the end of this section,  we consider the derivation extension of the multiplicative hom 3-Lie algebra $(L, [\cdot, \cdot,, \cdot]_L, \alpha)$ and
 give an application of the $\alpha$-derivation $\Der_{\alpha}(L)$.

For any linear map $D:L \rightarrow L$,  consider the vector space $L \oplus \mathbb{R}D$. Define a 3-ary skew-symmetric operation $[\cdot, \cdot, \cdot]_{L}$ on $L \oplus \mathbb{R}D$ by
$$[u+lD, v+mD,w+nD]_{D}=[u, v,w,]_{L}+lD(v)+mD(w)-nD(u), $$
$$ [u, v,w]_{D}=[u, v,w]_{L}, [D, u,v]_{D}=[v, D,u]_{D}=-[ u,v,D]_{D}=D(u), \forall u, v, w\in L.$$
Define a linear map $\alpha^{'}:L \oplus \mathbb{R}D \rightarrow L \oplus \mathbb{R}D$ by $\alpha^{'}(u+D)=\alpha(u)+D, $ i.e.\\
\begin{displaymath}
\mathbf{\alpha^{'}} =
\left( \begin{array}{cc}
\alpha & 0 \\
0 & 1  \\
\end{array} \right).
\end{displaymath}
\begin{thm}
With the above notations, $(L \oplus \mathbb{R}D,  [\cdot, \cdot, \cdot]_{D}, \alpha^{'})$ is a multiplicative hom 3-Lie algebra if and only if D is an $\alpha$- derivation of the multiplicative hom 3-Lie algebra $(L, [\cdot, \cdot, \cdot]_L, \alpha)$.
\end{thm}
\begin{proof}First, for any $u,  v,  w\in L$, $l,  m, n\in \mathbb{R}$, we have
\begin{eqnarray*}
&&\alpha^{'}([u+lD, v+mD,w+nD]_{D})=\alpha^{'}([u, v,w,]_{L}+lD(v)+mD(w)-nD(u))\\
&=&\alpha([u, v,w]_{L})+l\alpha\circ D(v)+m\alpha\circ
D(w)-n\alpha\circ D(u)
\end{eqnarray*}
and
\begin{eqnarray*}
&&[\alpha^{'}(u+lD), \alpha^{'}(v+mD),\alpha^{'}(w+nD)]_{D}=[\alpha(u)+lD, \alpha(v)+mD,\alpha(w)+nD]_{D}\\
&=&[\alpha(u), \alpha(v),\alpha(w)]_{L}+lD\circ \alpha(v)+mD\circ \alpha(w)-nD\circ \alpha(u).
\end{eqnarray*}
Since $\alpha$ is an algebra morphism,  $\alpha^{'}$ is an algebra
morphism if and only if
$$D\circ \alpha=\alpha\circ D.$$
 On the other hand,  we have
\begin{eqnarray*}
&&[\alpha^{'}(x),\alpha^{'}(D), [u, v,w]_{D}]_{D}=[\alpha(x),D, [u, v,w]_{D}]_{D}\\
&=&[[x, D,u]_{D},\alpha(v),\alpha(w)]_{D}+[\alpha(u),[x, D,v]_{D},\alpha(w)]_{D}+[\alpha(u),\alpha(v),[x, D,w]_{D}]_{D}\\
&=&[D(u),\alpha(v),\alpha(w)]_{D}+[\alpha(u),D(v),\alpha(w)]_{D}+[\alpha(u),\alpha(v),D(w)]_{D}.
\end{eqnarray*}
 It is obvious that the hom-Jacobi identity is satisfied  if and only if the following condition holds
$$D[u, v,w]_{D}=[D(u),\alpha(v),\alpha(w)]_{D}+[\alpha(u),D(v),\alpha(w)]_{D}+[\alpha(u),\alpha(v),D(w)]_{D}.$$
So $(L \oplus RD,  [\cdot, \cdot, \cdot]_{D}, \alpha^{'})$ is a
multiplicative hom 3-Lie algebra if and only if $D$ is an
$\alpha$-derivation of $(L, [\cdot, \cdot, \cdot]_L, \alpha)$.
\end{proof}

\section{$T_{\theta}$-extension of hom 3-Lie algebras}

In this section, we study representations of  multiplicative hom
3-Lie algebras. We also prove that one can form semidirect product
multiplicative hom 3-Lie algebras when given  representations of
multiplicative hom 3-Lie algebras. Let $(L, [\cdot, \cdot, \cdot]_L,
\alpha)$ be a multiplicative hom 3-Lie algebra and $V$ be an vector
space. Let $A\in gl(V)$ be an arbitrary linear transformation from
$V$  to $V$.
\begin{defn}
A representation of  multiplicative hom 3-Lie algebra  $(L, [\cdot, \cdot, \cdot]_L, \alpha)$ on the vector space V with respect to  $A\in gl(V)$ is a bilinear map $\rho_ A:L \rightarrow gl(V)$,  such that for any $ x,y,z,u\in L$,  the basic property of this map is that:
\begin{eqnarray}
&&\rho_{A}(\alpha(u),\alpha(v))\circ A=A\circ\rho_{A}(u,v),\\
&&\rho_{A}([x,y,z],\alpha(u))\circ A\notag\\
&&=\rho_{A}(\alpha(y),\alpha(z))\rho_{A}(x,u)+\rho_{A}(\alpha(z),\alpha(x))\rho_{A}(y,u)+\rho_{A}(\alpha(x),\alpha(y))\rho_{A}(z,u),\\
&&\rho_{A}(\alpha(x),\alpha(y))\rho_{A}(z,u)\notag\\
&&=\rho_{A}(\alpha(z),\alpha(u))\rho_{A}(x,y)+\rho_{A}([x,y,z],\alpha(u))\circ A+\rho_{A}(\alpha(z),[x,y,u])\circ A.
\end{eqnarray}
then $(V,\rho_{A})$ is called a representation of $L$, or $V$ is an $L$-module.

As an example, let $\rho_{A}(x,y)=\ad(x,y)$ for all $x,y\in L$. Then $(L,\ad)$ is an $L$-module, called the adjoint module of $L$. If $(V,\rho_{A})$ is an L-module, then the dual space $V^{\ast}$ of $V$ is an $L$-module in the following way. For $f\in V^{\ast}, v\in V, x,y,\in L$, define $\rho_{A}^{\ast}:L\wedge L\longrightarrow End(V^{*}).$
 \begin{eqnarray}\rho_{A}^{\ast}(x,y)(f)(v)=-f(\rho_{A}(x,y)(v)).\end{eqnarray}
 We call the $V^{*}$ the dual module of  $V$.
\end{defn}

In the case of Lie algebras, we can form semidirect products when given representations. Similarly,  we have
\begin{prop}
Given a representation $\rho _A$ of the multiplicative hom 3-Lie algebra $(L, [\cdot, \cdot, \cdot]_L, \alpha)$ on the vector space V with respect to $A\in gl(V) $. Define a 3-ary skew-symmetric bracket operation $[\cdot, \cdot, \cdot]_{\rho A}:\wedge^{3}(L \oplus V) \rightarrow L \oplus V$ by
\begin{equation}
[u+X, v+Y,w+Z]_{\rho A}=[u, v,w]_{L}+\rho_{A}(u,v)(Z)+\rho_{A}(w,u)(Y)+\rho_{A}(v,w)(X),
\end{equation}
Define $\alpha+A:L\oplus V \rightarrow L \oplus V$ by
$$(\alpha+A)(u+X)=\alpha(u)+AX.$$
Then $(L\oplus V, [\cdot, \cdot, \cdot]_{\rho A}, \alpha+A)$ is a multiplicative hom 3-Lie algebra,  which we call the
semidirect product of the multiplicative hom 3-Lie algebra $(L, [\cdot, \cdot, \cdot]_L, \alpha)$ and V.
\end{prop}
\begin{proof}
First we show that $\alpha+A$ is an algebra morphism, On one hand, we have
\begin{eqnarray*}
&&(\alpha+A)([u+X, v+Y,w+Z]_{\rho A})\\
&=&(\alpha+A)([u, v,w]_{L}+\rho_{A}(u,v)(Z)+\rho_{A}(w,u)(Y)+\rho_{A}(v,w)(X))\\
&=&\alpha([u, v,w]_{L})+A\circ\rho_{A}(u,v)(Z)+A\circ\rho_{A}(w,u)(Y)+A\circ\rho_{A}(v,w)(X).
\end{eqnarray*}
On the other hand,  we have
\begin{eqnarray*}
&&[(\alpha+A)(u+X), (\alpha+A)(v+Y),(\alpha+A)(w+Z)]_{\rho A}\\
&=&[\alpha(u)+AX, \alpha(v)+AY, \alpha(w)+AZ]_{\rho_{A}}\\
&=&[\alpha(u), \alpha(v),\alpha(w)]_{L}+\rho_{A}(\alpha(u),\alpha(v))(AZ)+A\circ\rho_{A}(\alpha(w),\alpha(u))(AY)\\
&&+A\circ\rho_{A}(\alpha(v),\alpha(w))(AX).
\end{eqnarray*}
Since $\alpha$ is an algebra morphism,  $\rho_A$ and A satisfy
$(9)$,  it follows that $\alpha+A$ is a morphism with respect to the
bracket to $[\cdot, \cdot,\cdot]_{\rho A}$. It is not hard to deduce
that
\begin{eqnarray*}
&&[(\alpha+A)(v_{1}+Y_{1}), (\alpha+A)(v_{2}+Y_{2}),[u_{1}+X_{1}, u_{2}+X_{2},u_{3}+X_{3}]_{\rho A}]_{\rho A}\\
&=&[\alpha (v_{1}), \alpha(v_{2}),[u_{1}, u_{2},u_{3}]_{L}]_{L}+\rho_{A}(\alpha (v_{1}), \alpha(v_{2}))\rho_{A}(u_{1}, u_{2})(X_{3})\\
&&+\rho_{A}(\alpha (v_{1}), \alpha(v_{2}))\rho_{A}(u_{3}, u_{1})(X_{2})+\rho_{A}(\alpha (v_{1}), \alpha(v_{2}))\rho_{A}(u_{2}, u_{3})(X_{1})\\
&&+\rho_{A}([u_{1}, u_{2},u_{3}]_{L},\alpha(v_{1}))\circ
A(Y_{2})+\rho_{A}(\alpha (v_{2}),[u_{1}, u_{2},u_{3}]_{L})\circ
A(Y_{1}),
\end{eqnarray*}
\begin{eqnarray*}
&&[[v_{1}+Y_{1},v_{2}+Y_{2},u_{1}+X_{1}]_{\rho A},(\alpha+A)(u_{2}+X_{2}),(\alpha+A)(u_{3}+X_{3})]_{\rho A}\\
&=&[[v_{1}, v_{2},u_{1}]_{L},\alpha(u_{2}),\alpha(u_{3})]_{L}+\rho_{A}([v_{1},v_{2},u_{1}]_{L},\alpha(u_{2}))\circ A(X_{3})\\
&&+\rho_{A}(\alpha(u_{3}),[v_{1}, v_{2},u_{1}]_{L})\circ A(X_{2})+\rho_{A}(\alpha(u_{2}),\alpha(u_{3}))\rho_{A}(v_{1},v_{2})(X_{1})\\
&&+\rho_{A}(\alpha(u_{2}),\alpha(u_{3}))\rho_{A}(u_{1},v_{1})(Y_{2})+\rho_{A}(\alpha(u_{2}),\alpha(u_{3}))\rho_{A}(v_{2},u_{1})(Y_{1}),
\end{eqnarray*}
\begin{eqnarray*}
&&[(\alpha+A)(u_{1}+X_{1}),[v_{1}+Y_{1},v_{2}+Y_{2},u_{2}+X_{2}]_{\rho A},(\alpha+A)(u_{3}+X_{3})]_{\rho A}\\
&=&[\alpha(u_{1}),[v_{1},v_{2},u_{2}]_{L},\alpha(u_{3})]_{\rho_A}+\rho_{A}(\alpha(u_{1}),[v_{1},v_{2},u_{2}]_{L})\circ A(X_{3})\\
&&+\rho_{A}([v_{1},v_{2},u_{2}]_{L},\alpha(u_{3}))\circ A(X_{1})+\rho_{A}(\alpha(u_{3}),\alpha(u_{1}))\rho_{A}(v_{1},v_{2})(X_{2})\\
&&+\rho_{A}(\alpha(u_{3}),\alpha(u_{1}))\rho_{A}(u_{2},v_{1})(Y_{2})+\rho_{A}(\alpha(u_{3}),\alpha(u_{1}))\rho_{A}(v_{2},u_{2})(Y_{1}),
\end{eqnarray*}
\begin{eqnarray*}
&&[(\alpha+A)(u_{1}+X_{1}),(\alpha+A)(u_{2}+X_{2}),[v_{1}+Y_{1},v_{2}+Y_{2},u_{3}+X_{3}]_{\rho A}]_{\rho A}\\
&=&[\alpha(u_{1}),\alpha(u_{2}),[v_{1},v_{2},u_{3}]_{L}]_{\rho_A}+\rho_{A}(\alpha(u_{1}),\alpha(u_{2}))\rho_{A}(v_{1},v_{2})(X_{3})\\
&&+\rho_{A}(\alpha(u_{1}),\alpha(u_{2}))\rho_{A}(u_{3},v_{1})(Y_{2})+\rho_{A}(\alpha(u_{1}),\alpha(u_{2}))\rho_{A}(v_{2},u_{3})(Y_{1})\\
&&+\rho_{A}([v_{1},v_{2},u_{3}]_{L},\alpha(u_{1}))\circ
A(X_{2})+\rho_{A}(\alpha(u_{2}),[v_{1},v_{2},u_{2}]_{L})\circ
A(X_{1}).
\end{eqnarray*}
By (11), we have
\begin{eqnarray*}
&&\rho_{A}(\alpha (v_{1}), \alpha(v_{2}))\rho_{A}(u_{1},u_{2})(X_{3})\\
&=&\rho_{A}(\alpha(u_{1}),\alpha(u_{2}))\rho_{A}(v_{1},v_{2})(X_{3})+\rho_{A}([v_{1},v_{2},u_{1}]_{L},\alpha(u_{2}))\circ A(X_{3})\\
&&+\rho_{A}(\alpha(u_{1}),[v_{1},v_{2},u_{2}]_{L})\circ A(X_{3}).
\end{eqnarray*}
Similar, we have
 \begin{eqnarray*}
 &&[(\alpha+A)(v_{1}+Y_{1}), (\alpha+A)(v_{2}+Y_{2}),[u_{1}+X_{1}, u_{2}+X_{2},u_{3}+X_{3}]_{\rho A}]_{\rho A}\\
 &=&[[v_{1}+Y_{1},v_{2}+Y_{2},u_{1}+X_{1}]_{\rho A},(\alpha+A)(u_{2}+X_{2}),(\alpha+A)(u_{3}+X_{3})]_{\rho A}\\
 &&+[(\alpha+A)(u_{1}+X_{1}),[v_{1}+Y_{1},v_{2}+Y_{2},u_{2}+X_{2}]_{\rho A},(\alpha+A)(u_{3}+X_{3})]_{\rho A}\\
 &&+[(\alpha+A)(u_{1}+X_{1}),(\alpha+A)(u_{2}+X_{2}),[v_{1}+Y_{1},v_{2}+Y_{2},u_{3}+X_{3}]_{\rho A}]_{\rho A}.
 \end{eqnarray*}
The hom-Jacobi identity is satisfied. Thus,  $(L\oplus V, [\cdot, \cdot, \cdot]_{\rho_A}, \alpha+A)$ is a multiplicative hom 3-Lie algebra.
\end{proof}

\begin{defn}
Let $(L, [\cdot, \cdot, \cdot]_{L}, \alpha)$ be a hom 3-Lie algebra and $(V ,\rho_A)$ be an $L$-module. If $\theta:L\wedge L\wedge L\rightarrow V$ is a 3-linear mapping and satisfies, for every $x,y,z,u,v\in L,$
 \begin{eqnarray}
  &&\theta([x,u,v]_{L},\alpha(y),\alpha(z))+\theta([y,u,v]_{L},\alpha(z),\alpha(x))+\theta(\alpha(x),\alpha(y),[z,u,v]_{L})\notag\\
  &&-\theta([x,y,z]_{L},\alpha(u),\alpha(v))+\rho_{A}(\alpha(y),\alpha(z))\theta(x,u,v)+\rho_{A}(\alpha(z),\alpha(x))\theta(y,u,v)\notag\\
  &&+\rho_{A}(\alpha(x),\alpha(y))\theta(z,u,v)-\rho_{A}(\alpha(u),\alpha(v))\theta(x,y,z)=0.
  \end{eqnarray}
  then $\theta$ is called a $3$-cocycle associated with $\rho_{A}$.
\end{defn}

Cocycles are a fundamental concept in the theory of cohomology of Lie algebras. Using $3$-cocycles, we are able to construct new hom 3-Lie algebras by adding $L$-modules as follow.
\begin{lem}Let $(L, [\cdot, \cdot, \cdot]_{L}, \alpha)$ be a hom 3-Lie algebra and $(V ,\rho_A)$ be an $L$-module. Define $\alpha+A:L\oplus V \rightarrow L \oplus V$ by
$$(\alpha+A)(x+f)=\alpha(x)+A\circ f,$$
and $f\in V$ is compatible with $\alpha$ and A in sense that $A\circ f=f\circ \alpha$.
If $\theta:L\wedge L\wedge L\rightarrow V$ is a $3$-cocycle, then $(L\oplus V,[\cdot,\cdot,\cdot]_{\theta},\alpha+A)$ is a hom 3-Lie algebra under the following multiplication:
\begin{eqnarray}&&[x_{1}+y_{1},x_{2}+y_{2},x_{3}+y_{3}]_{\theta}\notag\\
&=&[x_{1},x_{2},x_{3}]_{L}+\theta(x_{1},x_{2},x_{3})+\rho_A(x_{1},x_{2})(y_{3})+\rho_A(x_{3},x_{1})(y_{2})+\rho_A(x_{2},x_{3})(y_{1}).\end{eqnarray}
where $x_{1},x_{2},x_{3}\in L$ and $y_{1},y_{2},y_{3}\in V.$
\end{lem}

\begin{proof} It is easy by virtue of a routine computation.  \end{proof}
\begin{defn}
Denote the hom 3-Lie algebra $(L\oplus V,[\cdot,\cdot,\cdot]_{\theta},\alpha+A)$ by $T_{\theta}(L)$, $T_{\theta}(L)$ is called the $T_{\theta}$-extension of $(L, [\cdot, \cdot, \cdot]_{L}, \alpha)$ by the $L$-module $V$.
\end{defn}
\begin{lem}
Let $(L, [\cdot, \cdot, \cdot]_{L}, \alpha)$ be a hom 3-Lie algebra
and $(V ,\rho_A)$ be an $L$-module. For every linear mapping
$f:L\rightarrow V$, the ternary skew-symmetric mapping
$\theta_{f}:L\wedge L\wedge L\rightarrow V$ given by
\begin{eqnarray}\theta_{f}(x,y,z)=f([x,y,z]_{L})-\rho_A(x,y)f(z)-\rho_A(z,x)f(y)-\rho_A(y,z)f(x),\end{eqnarray}
for all $x,y,z\in L$, is a $3$-cocycle associated with $\rho_{A}$.
\end{lem}
\begin{proof}
A tedious caculation shows that, for every $x,y,z,u,v\in L$,
\begin{eqnarray}
  &&\theta_{f}([x,u,v]_{L},\alpha(y),\alpha(z))+\theta_{f}([y,u,v]_{L},\alpha(z),\alpha(x))+\theta_{f}(\alpha(x),\alpha(y),[z,u,v]_{L})\notag\\
  &&-\theta_{f}([x,y,z]_{L},\alpha(u),\alpha(v))+\rho_{A}(\alpha(y),\alpha(z))\theta_{f}(x,u,v)+\rho_{A}(\alpha(z),\alpha(x))\theta_{f}(y,u,v)\notag\\
  &&+\rho_{A}(\alpha(x),\alpha(y))\theta_{f}(z,u,v)-\rho_{A}(\alpha(u),\alpha(v))\theta_{f}(x,y,z)=0.
  \end{eqnarray}
  It follows that $\theta$ is called a $3$-cocycle associated with
  $\rho_{A}$.
\end{proof}
\begin{thm}
Let $(L, [\cdot, \cdot, \cdot]_{L}, \alpha)$ be a hom 3-Lie algebra
and $(V ,\rho_A)$ be an $L$-module. then for every linear mapping
$f:L\rightarrow V$,$$\sigma:T_{\theta}(L)\rightarrow
T_{\theta+\theta_{f}}(L),\sigma(x+v)=x+f(x)+v,\quad x\in L,v\in V,
$$ is a hom 3-Lie algebra isomorphism.
\end{thm}
\begin{proof}It is clear that $\sigma$ is a bijection. Next, for every $x_{1},x_{2},x_{3}\in L$, $v_{1},v_{2},v_{3}\in V$,
\begin{eqnarray*}
&&\sigma([x_{1}+v_{1},x_{2}+v_{2},x_{3}+v_{3}]_{\theta})\\
&=&\sigma([x_{1},x_{2},x_{3}]_{L}+\theta(x_{1},x_{2},x_{3})+\rho_A(x_{1},x_{2})(v_{3})+\rho_A(x_{3},x_{1})(v_{2})+\rho_A(x_{2},x_{3})(v_{1}))\\
&=&[x_{1},x_{2},x_{3}]_{L}+\theta(x_{1},x_{2},x_{3})+f([x_{1},x_{2},x_{3}]_{L})+\rho_A(x_{1},x_{2})(v_{3})+\rho_A(x_{3},x_{1})(v_{2})\\
&&+\rho_A(x_{2},x_{3})(v_{1}).
 \end{eqnarray*}
 On the other hand
 \begin{eqnarray*}
&&[\sigma(x_{1}+v_{1}),\sigma(x_{2}+v_{2}),\sigma(x_{3}+v_{3})]_{\theta+\theta_{f}}\\
&=&[x_{1}+f(x_{1})+v_{1},x_{2}+f(x_{2})+v_{2},x_{3}+f(x_{3})+v_{3}]_{\theta+\theta_{f}}\\
&=&[x_{1},x_{2},x_{3}]_{L}+(\theta+\theta_{f})(x_{1},x_{2},x_{3})+\rho_A(x_{1},x_{2})(f(x_{3}))+\rho_A(x_{1},x_{2})(v_{3})\\
&&+\rho_A(x_{3},x_{1})(f(x_{2}))+\rho_A(x_{3},x_{1})(v_{2})+\rho_A(x_{2},x_{3})(f(x_{1}))+\rho_A(x_{2},x_{3})(v_{1}).
 \end{eqnarray*}
 By (16), we have
$$\sigma([x_{1}+v_{1},x_{2}+v_{2},x_{3}+v_{3}]_{\theta})=[\sigma(x_{1}+v_{1}),\sigma(x_{2}+v_{2}),\sigma(x_{3}+v_{3})]_{\theta+\theta_{f}}.$$
\end{proof}

\section{$T_{\theta}^{*}(L)$-extension of metric hom 3-Lie algebras}
The method of T*-extension
has already been used for 3-Lie algebras in [8, 14]. Now we will generalize it to hom 3-Lie algebras.

Note that every  hom 3-Lie algebra $(L, [\cdot, \cdot, \cdot]_{L}, \alpha)$ is an $L$-module with respect to the adjoint representation $\ad:L\wedge L\rightarrow End(L)$, $\ad(x,y)(z)=[x,y,z]_{L}$ for every $x,y,z\in L$. By (12), $\ad^{*}(x,y):L\wedge L\rightarrow End(L^{*})$ is given by
 \begin{eqnarray}
 \ad^{*}(x,y)(f)(z)=-f([x,y,z]_{L}),
  \end{eqnarray}
for all $ x,y,z\in L$, $f\in L^{*}$, then $(L^{*}, {\rm ad}^{*})$ is
an $L$-module, called the coadjoint module of $L$. It follows from
Lemma 4.4  that, for each $3$-cocycle $\theta:L\wedge L\wedge
L\rightarrow L^{*}$ associated with $\ad^{*}$, the extension
$T_{\theta}^{*}(L)=L\oplus L^{\ast}$ provided with the following
bracket and linear map defined respectively by
\begin{eqnarray}
[x_{1}+f_{1},x_{2}+f_{2},x_{3}+f_{3}]_{\theta}&=&[x_{1},x_{2},x_{3}]_{L}+\theta(x_{1},x_{2},x_{3})+\ad^{*}(x_{1},x_{2})(f_{3})\notag\\
&&+\ad^{*}(x_{3},x_{1})(f_{2})+\ad^{*}(x_{2},x_{3})(f_{1}),
\end{eqnarray}
 $$\alpha^{'}(x+f)=\alpha(x)+f\circ \alpha,   \forall x_{i}\in L, f_{i}\in L^{*},i=1,2,3.$$
 Then $(L\oplus L^{\ast},[\cdot, \cdot, \cdot]_{\theta},\alpha^{'})$ is a hom 3-Lie algebra.
\begin{defn}
The hom 3-Lie algebra $(L\oplus L^{\ast},[\cdot, \cdot, \cdot]_{\theta},\alpha^{'})$ is called the $T_{\theta}^{*}$-extension of $(L, [\cdot, \cdot, \cdot]_{L}, \alpha)$.
\end{defn}
\begin{defn}
Let $L$ be a hom 3-Lie algebra over a field $\K$. We inductively define a derived series
$$(L^{(n)})_{n\geq 0}: L^{(0)}=L,\ L^{(n+1)}=[L^{(n)},L^{(n)},L],$$
a central descending series
$$(L^{n})_{n\geq 0}: L^{0}=L,\ L^{n+1}=[L^{n},L,L],$$
and a central ascending series
$$(C_{n}(L))_{n\geq 0}: C_{0}(L)=0, C_{n+1}(L)=C(C_{n}(L)),$$
where $C(I)=\{a\in L| [a,L,L]\subseteq I\}$ for a subspace $I$ of
$L$.

$L$ is called solvable and nilpotent(of length $k$) if and only if
there is a smallest integer $k$ such that $L^{(k)}=0$ and $L^{k}=0$,
respectively.
\end{defn}
\begin{thm}
Let $(L, [\cdot, \cdot, \cdot]_{L}, \alpha)$ be a hom 3-Lie algebra.
If $L$ is solvable, then every $T_{\theta}^{*}(L)$-extension of $(L,
[\cdot, \cdot, \cdot]_{L}, \alpha)$ is solvable. If $L$ is
nilpotent, so is $T_{\theta}^{*}(L)$.
\end{thm}
\begin{proof}
Let $(L, [\cdot, \cdot, \cdot]_{L}, \alpha)$ be solvable and
$L^{(s)}=[L^{(s-1)},L^{(s-1)},L]=0$ for some nonnegative integer
$s$. For each $3$-cocycle $\theta:L\wedge L\wedge L\rightarrow
L^{*}$ associated with $\ad^{*}$, by (19), we have
$$T_{\theta}^{*}(L)^{(1)}\subseteq[L, L, L]_{L}+[L, L, L^{*}]_{\theta}=L^{1}+[L, L, L^{*}]_{\theta}\subseteq L^{1}+L^{*}$$
and $f(L)\subseteq f(L^{1})$ for $f\in [L, L, L^{*}]_{\theta}$. Inductively, we have
$$T_{\theta}^{*}(L)^{(k)}\subseteq [T_{\theta}^{*}(L)^{(k-1)}, T_{\theta}^{*}(L)^{(k-1)}, L\oplus L^{*}]_{\theta}=L^{(k)}+ L^{*} .$$
Since $L^{(s)}=0$ and $L^{*}$ is an abelian ideal of $T_{\theta}^{*}(L)$, we obtain that $T_{\theta}^{*}(L)^{(s+1)}=0$. It follows
 that $T_{\theta}^{*}(L)$ is a solvable hom 3-Lie algebra.

If $(L, [\cdot, \cdot, \cdot]_{L}, \alpha)$ is nilpotent with $L^{s}=0$, a similar discussion yields that
$$T_{\theta}^{*}(L)^{(1)}\subseteq[L, L, L]_{L}+[L^{*},L, L]_{\theta},$$
$$\cdots$$
$$T_{\theta}^{*}(L)^{(k+1)}\subseteq L^{k}+[L,L,[\cdots[L,L,L^{*}]_{\theta}]_{\theta},\cdots]_{\theta}.$$
Note that for $f\in [L,L,[\cdots[L,L,L^{*}]_{\theta}]_{\theta},\cdots]_{\theta}$, we have $f(L)\subseteq f(L^{k})$. Thus $T_{\theta}^{*}(L)^{s}=0$,
 this is, $T_{\theta}^{*}(L)$ is a nilpotent hom 3-Lie algebra.
\end{proof}
Let $T_{\theta}^{*}(L)=L\oplus L^{*}$ be the $T_{\theta}^{*}$-extension of $L$. Define a bilinear form $q_{L}:T_{\theta}^{*}(L)\wedge T_{\theta}^{*}(L)\rightarrow F$ by
 \begin{eqnarray}
q_{L}(x_{1}+f_{1},x_{2}+f_{2})=f_{1}(x_{2})+f_{2}(x_{1}),\forall
x_{i}\in L,f_{i}\in L^{*},i=1,2,
  \end{eqnarray}
  this bilinear form is non-degenerate.
\begin{thm}
$(T_{\theta}^{*}(L),q_{L},\alpha^{'})$ is a metric hom 3-Lie algebra if and only if $\theta$ satisfies
 \begin{eqnarray}
\theta(x_{1},x_{2},x_{3})(x_{4})+\theta(x_{1},x_{2},x_{4})(x_{3})=0,\forall x_{i}\in L,i=1,2,3,4.
  \end{eqnarray}
\end{thm}
\begin{proof}
If $(T_{\theta}^{*}(L),q_{L},\alpha^{'})$ is a metric hom 3-Lie
algebra, it follows from (19) and (20) that $\theta$ satisfies (21).
Conversely, if $\theta$ satisfies (21), then by (18), for every
$x_{i}\in L,f_{i}\in L^{*},i=1,2,3,4$, we have
\begin{eqnarray*}
&&q_{L}([x_{1}+f_{1},x_{2}+f_{2},x_{3}+f_{3}]_{\theta},x_{4}+f_{4})+q_{L}([x_{1}+f_{1},x_{2}+f_{2},x_{4}+f_{4}]_{\theta},x_{3}+f_{3})\\
&=&q_{L}([x_{1},x_{2},x_{3}]_{L}+\theta(x_{1},x_{2},x_{3})+\ad^{*}(x_{1},x_{2})(f_{3})+\ad^{*}(x_{3},x_{1})(f_{2})\\
&&+\ad^{*}(x_{2},x_{3})(f_{1}),x_{4}+f_{4})+q_{L}([x_{1},x_{2},x_{4}]_{L}+\theta(x_{1},x_{2},x_{4})\\
&&+\ad^{*}(x_{1},x_{2})(f_{4})+\ad^{*}(x_{4},x_{1})(f_{2})+\ad^{*}(x_{2},x_{4})(f_{1}),x_{3}+f_{3})\\
&=&\ad^{*}(x_{1},x_{2})(f_{3})(x_{4})+\ad^{*}(x_{3},x_{1})(f_{2})(x_{4})+\ad^{*}(x_{2},x_{3})(f_{1})(x_{4})\\
&&+f_{4}([x_{1},x_{2},x_{3}]_{L})+\theta(x_{1},x_{2},x_{3})(x_{4})+\ad^{*}(x_{1},x_{2})(f_{4})(x_{3})+\ad^{*}(x_{4},x_{1})(f_{2})(x_{3})\\
&&+\ad^{*}(x_{2},x_{4})(f_{1})(x_{3})+f_{3}([x_{1},x_{2},x_{4}]_{L})+\theta(x_{1},x_{2},x_{4})(x_{3})\\
&=&\theta(x_{1},x_{2},x_{3})(x_{4})+\theta(x_{1},x_{2},x_{4})(x_{3}).
\end{eqnarray*}
 Thus $q_{L}$ is ad-invariant if and only if the identity (21) holds.
\end{proof}
\begin{defn}
Let $(G,B,\beta)$ and $(G^{'},B^{'},\beta^{'})$ be two metric hom
3-Lie algebras. If there exists a linear isomorphism $\sigma: G
\rightarrow G^{'}$ such that\\
$\sigma([x,y,z]_{G})=[\sigma(x),\sigma(y),\sigma(z)]_{G^{'}}$ and
 $B(x,y)=B^{'}(\sigma(x),\sigma(y))$
  for all $x,y\in G.$\\
Then $(G,B,\beta)$ and $(G^{'},B^{'},\beta^{'})$ are called isometric.
\end{defn}

 Next, we describe the relationship between metric hom 3-Lie algebras of even dimensions and the $T_{\theta}^{*}$-extension of hom 3-Lie algebras.
 \begin{thm}
 Let $(G,B,\beta)$ be a $2k$-dimensional metric hom 3-Lie algebra over ${\bf F.}$  Then $(G,B,\beta)$ is isometric to a $T_{\theta}^{*}$-extension
 of a hom 3-Lie algebra $(L, [\cdot, \cdot, \cdot]_{L}, \alpha)$ if and only if $G$ contains a $k$-dimensional isotropic ideal $I$.
 Furthermore, $L$ is isomorphic to $G/I$.
 \end{thm}
\begin{proof}
If $(G,B,\beta)$ is isomorphic to some $T_{\theta}^{*}$-extension
$(T_{\theta}^{*}(L),q_{L},\alpha^{'})$ of a hom 3-Lie algebra $(L,
[\cdot, \cdot, \cdot]_{L}, \alpha)$ with an isometric isomorphism
$T_{\theta}^{*}(L)=L\oplus L^{\ast}\rightarrow G$, then
$I=\sigma(L^{\ast})$ is an isotropic ideal of $G$ and ${\rm
dim}I={\rm dim}L^{\ast}=k$. Also, $L\cong T_{\theta}^{*}/L^{*}\cong
\sigma(T_{\theta}^{*}(L))/\sigma(L^{*})=G/I$.

Conversely, let $I$ be a $k$-dimensional abelian isotropic ideal of
$G$. Denote by $L$ the quotient hom 3-Lie algebra $G/I$. Let
$\rho:G\rightarrow L$ be
 the canonical projection such that $\rho(x)=x+I=\bar{x}$ for every $x\in G$. We can choose an isotropic complementary vector subspace $L_{0}$ of
 $G$ such that $G=L_{0}\oplus I$ and $L_{0}^{\perp}=L_{0}$. Then $\phi=\rho|_{L_{0}}\rightarrow L$ is a linear isomorphism.

Denote by $\rho_{0}$ and $\rho_{1}$ the projection $G\rightarrow L_{0}$ and $G\rightarrow I$, respectively. Let
\begin{eqnarray} &&\delta: I \rightarrow L^{*}, \delta(i)(\bar{x})=B(i,x),  \forall i\in I, \bar{x}\in L.\end{eqnarray}
As $I$ is isotropic and $B$ is non-degenerate on $G$, we see that $\delta$ is a linear isomorphism from $I$ onto $L^{*}$ and satisfies
\begin{eqnarray*}
\delta([g_{1},g_{2},i]_{G})(\bar{g_{3}})&=&B([g_{1},g_{2},i]_{G},g_{3})\\
&=&-B([g_{1},g_{2},g_{3}]_{G},i)\\
&=&-\delta(i)(\overline{[g_{1},g_{2},g_{3}]_{G}})\\
&=&-\delta(i)([\bar{g_{1}},\bar{g_{2}},\bar{g_{3}}]_{G})\\
&=&\ad^{*}(\bar{g_{1}},\bar{g_{2}})(\delta(i))(\bar{g_{3}}),
\end{eqnarray*}
for all $g_{1},g_{2},g_{3}\in G, i\in I,$ where the
$\ad^{*}(x,y):L\wedge L\rightarrow {\rm End}(L^{*})$ is the
coadjoint representation of $L$. Therefore,
\begin{eqnarray}
\delta([g_{1},g_{2},i]_{G})=\ad^{*}(\bar{g_{1}},\bar{g_{2}})(\delta(i)).
\end{eqnarray}

Set $\theta:L\wedge L\wedge L\rightarrow L^{*}$. We have
\begin{eqnarray}
\theta(\bar{x_{1}},\bar{x_{2}},\bar{x_{3}})=\delta(\rho_{1}([x_{1},x_{2},x_{3}]_{G})),\forall x_{1},x_{2},x_{3}\in L_{0}.
\end{eqnarray}
Thanks to (22)-(24), for every $x,y,z,u,v\in L_{0},$
\begin{eqnarray}
  &&0=(\theta([\bar{x},\bar{u},\bar{v}]_{L},\alpha(\bar{y}),\alpha(\bar{z}))+\theta([\bar{y},\bar{u},\bar{v}]_{L},\alpha(\bar{z}),\alpha(\bar{x}))+\theta(\alpha(\bar{x}),\alpha(\bar{y}),[\bar{z},\bar{u},\bar{v}]_{L})\notag\\
  &&-\theta([\bar{x},\bar{y},\bar{z}]_{L},\alpha(\bar{u}),\alpha(\bar{v}))+\ad^{*}(\alpha(\bar{y}),\alpha(\bar{z}))\theta(\bar{x},\bar{u},\bar{v})+\ad^{*}(\alpha(\bar{z}),\alpha(\bar{x}))\theta(\bar{y},\bar{u},\bar{v})\notag\\
  &&+\ad^{*}(\alpha(x),\alpha(y))\theta(z,u,v)-\ad^{*}(\alpha(u),\alpha(v))\theta(x,y,z))(\bar{w}),
  \end{eqnarray}
which shows that $\theta$ is a $3$-cocycle associated with
$\ad^{*}$. Then we get a $T_{\theta}^{*}$-extension
$T_{\theta}^{*}(L)=L\oplus L^{*}$ of $L$ with the following
multiplication:
\begin{eqnarray}
[\bar{x_{1}}+f_{1},\bar{x_{2}}+f_{2},\bar{x_{3}}+f_{3}]_{\theta}&=&[\bar{x_{1}},\bar{x_{2}},\bar{x_{3}}]_{L}+\delta(\rho_{1}([x_{1},x_{2},x_{3}]_{G}))+\ad^{*}(\bar{x_{1}},\bar{x_{2}})(f_{3})\notag\\
&&+\ad^{*}(\bar{x_{3}},\bar{x_{1}})(f_{2})+\ad^{*}(\bar{x_{2}},\bar{x_{3}})(f_{1}),
\end{eqnarray} for all $x_{i}\in L_{0}$, $f_{i}\in L^{*}$,$i=1,2,3.$

Let $$\sigma:G= L_{0}\oplus I\rightarrow T_{\theta}^{*}(L),
\sigma(x+i)=x+\delta(i), \qquad \forall x\in L_{0},i\in I.$$ Then
$\sigma$ is a linear isomorphism, and for every
$x_{1},x_{2},x_{3}\in L_{0}$, $i_{1},i_{2},i_{3},\in I$,
\begin{eqnarray*}
&&\sigma([x_{1}+i_{1},x_{2}+i_{2},x_{3}+i_{3}]_{L})\\
&=&\sigma([x_{1},x_{2},x_{3}]_{L}+[i_{1},x_{2},x_{3}]_{L}+[x_{1},i_{2},x_{3}]_{L}+[x_{1},x_{2},i_{3}]_{L})\\
&=&\sigma(\rho_{0}([x_{1},x_{2},x_{3}]_{L})+\rho_{1}([x_{1},x_{2},x_{3}]_{L})+[i_{1},x_{2},x_{3}]_{L}+[x_{1},i_{2},x_{3}]_{L}+[x_{1},x_{2},i_{3}]_{L})\\
&=&\overline{[x_{1},x_{2},x_{3}]}_{L}+\delta(\rho_{1}([x_{1},x_{2},x_{3}]_{L})+[i_{1},x_{2},x_{3}]_{L}+[x_{1},i_{2},x_{3}]_{L}+[x_{1},x_{2},i_{3}]_{L})\\
&=&\overline{[x_{1},x_{2},x_{3}]}_{L}+\theta([x_{1},x_{2},x_{3}]_{L})+\delta([i_{1},x_{2},x_{3}]_{L}+[x_{1},i_{2},x_{3}]_{L}+[x_{1},x_{2},i_{3}]_{L})\\
&=&\overline{[x_{1},x_{2},x_{3}]}_{L}+\theta([x_{1},x_{2},x_{3}]_{L})+\ad^{*}(\bar{x_{1}},\bar{x_{2}})(\delta(i_{3}))+\ad^{*}(\bar{x_{3}},\bar{x_{1}})(\delta(i_{2}))\\
&&+\ad^{*}(\bar{x_{2}},\bar{x_{3}})(\delta(i_{1}))\\
&=&[\bar{x_{1}}+\delta(i_{1}),\bar{x_{2}}+\delta(i_{2}),\bar{x_{3}}+\delta(i_{3})]_{\theta}\\
&=&[\sigma(x_{1}+i_{1}),\sigma(x_{2}+i_{2}),\sigma(x_{3}+i_{3})]_{\theta}.
\end{eqnarray*}
It follows from the $\ad$-invariance of $B$ that for every $u,v,w,x\in
L_{0}$,
\begin{eqnarray*}
\theta(\bar{u},\bar{v},\bar{w})(\bar{x})+\theta(\bar{u},\bar{v},\bar{x})(\bar{w})&=&\delta(\rho_{1}([u,v,w]_{L})(\bar{x})+\delta(\rho_{1}([u,v,x]_{L})(\bar{u})\\
&=&B([u,v,w]_{L})(x)+B([u,v,x]_{L})(w)=0.
\end{eqnarray*}
Therefore, $(T_{\theta}^{*}(L),q_{L},\alpha^{'})$ is a metric hom 3-Lie algebra. Also, for every $x\in L_{0},i\in I$,
$$q_{L}(\sigma(x),\sigma(i))=q_{L}(\bar{x},\delta(i))=\delta(i)(\bar{x})=B(x,i).$$
Thus $(G,B,\beta)$ is isometric to $(T_{\theta}^{*}(L),q_{L},\alpha^{'}).$
\end{proof}

The proof of the above theorem tell us that the $3$-cocycle $\theta$ associated with $\ad^{*}$ depends on the choice of the
isotropic subspace $L_{0}$(of $G$) which is complementary to the ideal $I$. Different $T_{\theta}^{*}(L)$-extensions of a metric hom 3-Lie
algebra may lead to the same metric hom 3-Lie algebra.

\section{Conclusions}

This paper gives the definition of
  hom 3-Lie algebras and shows that the direct sum of two  hom 3-Lie algebras
  is still a  hom 3-Lie algebra.  For any nonnegative
   integer k,  we define $\alpha^{k}$-derivations of multiplicative  hom 3-Lie
   algebras study derivations of multiplicative  hom 3-Lie algebras.  In particular,  any $\alpha$-derivations
   gives rise to a derivation extension of the multiplicative  hom 3-Lie algebras
   $(L, [\cdot, \cdot, \cdot]_L, \alpha)$(Theorem 3.5). We also give the
   definition of representations of multiplicative hom 3-Lie algebras and show that
   one can obtain the semidirect product multiplicative hom 3-Lie algebras
   $(L\oplus V, [\cdot, \cdot, \cdot]_{\rho A}, \alpha+A)$ associated to any
   representation $\rho A$ on $V$ of the multiplicative hom 3-Lie algebras
   $(L, [\cdot, \cdot, \cdot]_L, \alpha)$ (Proposition 4.2). In the end of the paper, we describe
    module extensions (associated with a 3-cocycle) of hom 3-Lie algebras and study $T_{\theta}^{*}$-extensions (Definition 5.1)of  hom 3-Lie algebras.


\begin{thebibliography}{99}

\bibitem {1} Ammar F.  and  Makhlouf A., \textsl{Hom-Lie superalgebras and Hom-Lie admissible
superalgebras.} J. Algebra 324 (2010), no. 7, 1513-1528.

\bibitem {Faouzi&Abdenacer} Ammar F.  and  Makhlouf A., \textsl{Cohomology of Hom-Lie superalgebras and q- deforemed Witt superalgebra}. e-Print: arXiv:1204.6244.

\bibitem {Faouzi&Abdenacer} Ammar F., Ayadi I.  and  Makhlouf A., \textsl{Quadratic color Hom-Lie
algebras}. e-Print: arXiv:1204.5155.


\bibitem {Bordemann} Bordemann M. \textsl{Nondegenerate invariant bilinear forms on nonassociative algebras.} Acta Math. Univ. Comenianae LXVI(2)(1997), 151-201.

\bibitem {Bajo&Benayadi&Medina} Bajo I., Benayadi S., Medina A., \textsl{Symplectic structures on quadratic Lie algebras.} J. Algebra 316(2007), no. 1, 174-188.


\bibitem {1} Bagger J. and Lambert N.,  \textsl{Gauge symmetry and supersymmetry of multiple M2-branes.} Phys. Rev. D77 (2008), p. 065008.

\bibitem {1} Bagger J. and Lambert N., \textsl{Comments on multiple M2-branes.} JHEP 0802 (2008), 105

\bibitem {1} Bai, R., Wu, W., Li, Y., Li, Z,  \textsl{Module extensions of 3-Lie algebras.} Linear Multilinear Algebra 60 (2012), no. 4, 433-447.

\bibitem {Bordemann} Bordemann M. \textsl{Nondegenerate invariant bilinear forms on nonassociative algebras.} Acta Math. Univ. Comenianae LXVI(2)(1997), 151-201.

\bibitem {1} Gustavsson A.,  \textsl{Algebraic structure on parallel M2-branes}.  Available at arXiv: 0709.1260.

\bibitem {1} Gohr A.  \textsl{On hom-algebras with surjective twisting.} J. Algebra 324 (2010), no. 7, 1483-1491.


\bibitem {Hartwig & Larsson} Hartwig J., Larsson D.  and Silvestrov S., \textsl{Deformations of Lie algebras using $\sigma$-derivations.}  J. Algebra 295(2006), 321-344.

\bibitem {1} Jin Q. and Li X.,  \textsl{ Hom-Lie algebra structures on semi-simple Lie algebras.} J. Algebra 319 (2008), no. 4, 1398-1408.

\bibitem {} Lin J., Wang Y.,  Deng S.,\textsl{$T^{*}$-extension of Lie triple systems}. Linear Algebra and its Applications {431}(2009), 2071-2083.

\bibitem {1} Liu, W, Zhang, Z,  \textsl{T*-extension of a 3-Lie algebra.} Linear Multilinear Algebra 60 (2012), no. 5, 538-594.


\bibitem {Scheunert&Zhang} Scheunert M., Zhang R., \textsl{Cohomology of Lie superalgebras and their generalizations.} J. Math. Phys. 39(1998), 5024-5061.

\bibitem {Sheng&Yunhe} Sheng Y., \textsl{Representations of Hom-Lie Algebras.} Algebr. Represent. Theory, 15 (2012), no. 6, 1081-1098.

\bibitem {Yau&Donald} Yau D.,  \textsl{The Hom-Yang-Baxter equation and Hom-Lie algebras.} J. Math. Phys. 52 (2011), no. 5, 053502, 19 pp.


\bibitem {1}  Yau D.,  \textsl{Hom-quantum groups: I. Quasi-triangular Hom-bialgebras.} J. Phys. A 45 (2012), no. 6, 065203, 23 pp.

\end{thebibliography}
\end{document}